\documentclass[12pt]{article}
\usepackage[a4paper, margin=2.3cm]{geometry}
\usepackage{amsmath,amssymb,amsthm}
\usepackage{color}
\usepackage{parskip}

\newtheorem{theorem}{Theorem}[section]

\newtheorem{lemma}[theorem]{Lemma}
\newtheorem*{definition*}{Definition}

\newcommand\ZZ{\mathbb Z}
\newcommand\NN{\mathbb N}
\newcommand\FF{\mathbb F}
\newcommand{\GL}{\ensuremath{\operatorname{GL}}}
\newcommand\floor[1]{\left\lfloor #1\right\rfloor}

\renewcommand\vec{\underline}

\begin{document}
\title{A structure theorem for product sets in extra special groups}

\author{
Thang Pham\thanks{Department of Mathematics, EPFL, Lausanne, Switzerland. Email \texttt{thang.pham@epfl.ch}. Thang Pham was partially supported by Swiss National Science Foundation grants no. 200020-162884 and 200021-175977} \and
Michael Tait\thanks{Department of Mathematical Sciences, Carnegie Mellon University. Michael Tait was supported by NSF grant DMS-1606350. Email: \texttt{mtait@cmu.edu}.} \and
Le Anh Vinh\thanks{University of Education, Vietnam National University. Email: \texttt{vinhla@vnu.edu.vn}}\and
Robert Won\thanks{Department of Mathematics, Wake Forest University. Email: \texttt{wonrj@wfu.edu}}
}
\date{}
\maketitle
\begin{abstract}
Hegyv\'ari and Hennecart showed that if $B$ is a sufficiently large brick of a Heisenberg group, then the product set $B\cdot B$ contains many cosets of the center of the group.  We give a new, robust proof of this theorem that extends to all extra special groups as well as to a large family of quasigroups. 
\end{abstract}
\section{Introduction}
Let $p$ be a prime. An extra special group $G$ is a $p$-group whose center $Z$ is cyclic of order $p$ such that $G/Z$ is an elementary abelian $p$-group (nice treatments of extra special groups can be found in \cite{A, G}). The extra special groups have order $p^{2n+1}$ for some $n \geq 1$ and occur in two families. Denote by $H_n$ and $M_n$ the two non-isomorphic extra special groups of order $p^{2n+1}$. Presentations for these groups are given in \cite{D}:
\begin{align*} H_n = \langle &a_1, b_1, \dots, a_n, b_n, c \mid [a_i,a_j]=[b_i,b_j]=1, [a_i, b_j] = 1 \mbox{ for } i \neq j, \\
&[a_i, c] = [b_i, c]=1, [a_i, b_i] = c, a_i^p = b_i^p = c_i^p =1 \mbox{ for } 1 \leq i \leq n \rangle
\\ 
M_n = \langle &a_1, b_1, \dots, a_n, b_n, c \mid [a_i,a_j]=[b_i,b_j]=1, [a_i, b_j] = 1 \mbox{ for } i \neq j, \\
&[a_i, c] = [b_i, c]=1, [a_i, b_i] = c, a_i^p = c_i^p = 1, b_i^p = c \mbox{ for } 1 \leq i \leq n \rangle.
\end{align*}
From these presentations, it is not hard to see that the center of each of these groups consists of the powers of $c$ so are cyclic of order $p$. It is also clear that the quotient of both groups by their centers yield elementary abelian $p$-groups. 

In this paper we consider the structure of products of subsets of extra special groups. The structure of sum or product sets of groups and other algebraic structures has a rich history in combinatorial number theory. One seminal result is Freiman's theorem \cite{F}, which asserts that if $A$ is a subset of integers and $|A+A| = O(|A|)$, then $A$ must be a subset of a small generalized arithmetic progression. Green and Ruzsa \cite{GR} showed that the same result is true in any abelian group. On the other hand, commutativity is important as the theorem is not true for general non-abelian groups \cite{HH2}. With this in mind, Hegyv\'ari and Hennecart were motivated to study what actually can be said about the structure of product sets in non-abelian groups. 

The group $H_1$ is the classical Heisenberg group, so the groups $H_n$ form natural generalizations of the Heisenberg group.  The group $H_n$ has a well-known representation as a subgroup of $\GL_{n+2}(\FF_p)$ consisting of upper triangular matrices
\[ [\vec{x}, \vec{y}, z] :=  \begin{bmatrix} 1 & \vec{x} & z \\
0 & I_n & \vec{y} \\
0 & 0 & 1
\end{bmatrix}
\]
where $\vec{x}, \vec{y} \in \FF_p^n$, $z \in \FF_p$, and $I_n$ is the $n \times n$ identity matrix. Let $\vec{e_i} \in \FF_p^n$ be the $i \textsuperscript{th}$ standard basis vector. In the presentation for $H_n$, $a_i$ corresponds to $[\vec{e_i}, 0, 0]$, $b_i$ corresponds to $[0, \vec{e_i}, 0]$ and $c$ corresponds to $[0,0,1]$. By matrix multiplication, we have 
\[  [\vec{x}, \vec{y}, z] \cdot [\vec{x}', \vec{y}', z'] = [\vec{x} + \vec{x}', \vec{y} + \vec{y}', z + z' + \langle \vec{x},  \vec{y}' \rangle]
\]
where $\langle \, , \rangle$ denotes the usual dot product.

Let $H_n$ be a Heisenberg group.  A subset $B$ of $H_n$ is said to be a \textit{brick} if
$$
B=\{[\underline{x}, \underline{y},z]\text{ such that }\underline{x}\in\underline{X},\   \underline{y}\in\underline{Y},\ z\in Z\}
$$
where $\underline{X}=X_1\times\dots\times X_n$ and
$\underline{Y}=Y_1\times\dots\times Y_n$ with 
non empty-subsets $X_i,Y_i, Z\subseteq\mathbb{F}_p$.
\medskip

\begin{theorem}[\textbf{Hegyv\'ari-Hennecart}, \cite{HH}]\label{th*}
For every $\varepsilon>0,$ there exists a positive integer $n_0$ such that if $n\ge n_0$, $B\subseteq H_n$ is a brick and
$$
|B|>|H_n|^{3/4+\varepsilon}
$$ 
then there exists a non trivial subgroup $G'$ of $H_n$, namely its center $[\underline{0},\underline{0},\mathbb{F}_p]$, such that $B\cdot B$ contains  
at least $|B|/p$ cosets of   $G'$. 
\end{theorem}

Our main focus is to extend this theorem to the second family of extra special groups $M_n$. A second focus of this paper is to consider generalizations of the higher dimensional Heisenberg groups where entries come from a (left) quasifield $Q$ rather than $\FF_p$. We recall the definition of a (left) quasifield:

A set $L$ with a binary operation $*$ is called a \emph{loop} if 
\begin{enumerate}
\item the equation $a * x = b$ has a unique solution in $x$ for every $a,b \in L$,
\item the equation $y * a = b$ has a unique solution in $y$ for every $a,b \in L$, and
\item there is an element $e \in L$ such that $e * x = x * e = x$ for all $x \in L$.
\end{enumerate}
A \emph{(left) quasifield} $Q$ is a set with two binary operations $+$ and $*$ such that $(Q , + )$ is a group with additive identity 0, $(Q^* , * )$ is a loop where $Q^* = Q \backslash \{0 \}$, and the following three conditions hold:
\begin{enumerate}
\item $a * (b + c) = a * b + a * c$ for all $a,b,c \in Q$,
\item $0 * x = 0$ for all $x \in Q$, and 
\item the equation $a * x = b * x + c$ has exactly one solution for every $a,b,c \in Q$ with $a \neq b$.  
\end{enumerate}  

Throughout the rest of the paper we will use the term quasifield to mean left quasifield. Given a quasifield $Q$, we define $H_n(Q)$ by the set of elements 
\[
\{[\vec{x}, \vec{y}, z]: \vec{x} \in Q^n, \vec{y}\in Q^n, z\in Q\}
\]
and a multiplication operation defined by 
\[  [\vec{x}, \vec{y}, z] \cdot [\vec{x}', \vec{y}', z'] = [\vec{x} + \vec{x}', \vec{y} + \vec{y}', z + z' + \langle \vec{x}, \vec{y}' \rangle],
\]
where for $\vec{x},\vec{y}\in Q^n$, if $\vec{x}=(x_1,\cdots, x_n)$ and $\vec{y} = (y_1,\cdots, y_n)$, we define $\langle \vec{x},\vec{y}\rangle = \sum_{i=1}^n x_i*y_i$. Then $H_n(Q)$ is a quasigroup with an identity element (ie, a loop), and when $Q = \FF_p$ we have that $H_n(Q)$ is the $n$-dimensional Heisenberg group $H_n$.

\subsection{Statement of main results}

Our theorems are analogous to Hegyv\'ari and Hennecart's theorem for the groups $M_n$ and the quasigroups $H_n(Q)$. In particular, their structure result is true for all extra special groups. We will define what it means for a subset $B$ of $M_n$ to be a brick in Section \ref{description section}.

\medskip

\begin{theorem}\label{thm1}
For every $\varepsilon>0$, there exists a positive integer $n_0=n_0(\varepsilon)$ such that if $n\ge n_0$, $B\subseteq M_n$ is a brick and
$$
|B|>|M_n|^{3/4+\varepsilon}
$$ 
then there exists a non trivial subgroup $G'$ of $M_n$, namely its center, such that $B\cdot B$ contains  
at least $|B|/p$ cosets of   $G'$. 
\end{theorem}

\medskip

Combining Theorem \ref{th*} and Theorem \ref{thm1}, we have 

\begin{theorem}\label{thm2}
Let $G$ be an extra special group.  For every $\varepsilon>0,$ there exists a positive integer $n_0=n_0(\epsilon)$ such that if $n\ge n_0$, $B\subseteq G$ is a brick and
$$
|B|>|G|^{3/4+\varepsilon}
$$ 
then there exists a non trivial subgroup $G'$ of $G$, namely its center, such that $B\cdot B$ contains  
at least $|B|/p$ cosets of   $G'$. 
\end{theorem}
\medskip

For $Q$ a finite quasifield, we similarly define a subset $B \subseteq H_n(Q)$ to be a {\em brick} if 
\[ B=\{[\underline{x}, \underline{y},z]\text{ such that }\underline{x}\in\underline{X},\   \underline{y}\in\underline{Y},\ z\in Z\}
\]
where $\underline{X}=X_1\times\dots\times X_n$ and
$\underline{Y}=Y_1\times\dots\times Y_n$ with 
non empty-subsets $X_i,Y_i, Z\subseteq Q$.
\medskip

\begin{theorem}\label{thm3}
Let $Q$ be a finite quasifield of order $q$. For every $\varepsilon > 0$, there exists an $n_0 = n_0(\varepsilon)$ such that if $n\geq n_0$, $B\subseteq H_n(Q)$ is a brick, and 
\[
|B| > |H_n(Q)|^{3/4+\varepsilon},
\]
then there exists a non trivial subquasigroup $G'$ of $H_n(Q)$, namely its center $[\vec{0}, \vec{0}, Q]$ such that $B\cdot B$ contains at least $|B|/q$ cosets of $G'$.
\end{theorem}

Taking $Q = \FF_p$ gives Theorem \ref{th*} as a corollary.

\section{Preliminaries}
\subsection{A description of $M_n$}\label{description section}
We give a description of $M_n$ with which it is convenient to work. Define a group $G$ whose elements are triples $[\vec{x}, \vec{y}, z]$ where $\vec{x} = (x_1, \dots, x_n)$, $\vec{y} = (y_1,\dots, y_n)$, with $x_i, y_i, z \in \FF_p$ for $1 \leq i \leq n$. The group operation in $G$ is given by
\[  [\vec{x}, \vec{y}, z] \cdot [\vec{x}', \vec{y}', z'] = [\vec{x} + \vec{x}', \vec{y} + \vec{y}', z + z' + \langle \vec{x},\vec{y}'\rangle + f(\vec{y}, \vec{y}')]
\]
where the function $f: \ZZ^n \times \ZZ^n \rightarrow \NN$ is defined by 
\[ f\left( (y_1, \dots, y_n), (y_1', \dots, y_n') \right) = \sum_{i =1}^n \floor{ \frac{y_i \bmod{p} + y_i' \bmod{p}}{p} },
\]
where the notation $x\bmod{p} \in \{0,1,\cdots, p-1\}$ means to take the element $x\in \mathbb{F}_p$ and consider it as an integer. Concretely, $f$ counts the number of components where (after reducing mod $p$) $y_i + y_i' \geq p$. This is slight abuse of notation, as $\vec{y}, \vec{y}' \in \FF_p^n$, but is well-defined if we regard them as elements of $\ZZ^n$.

\begin{lemma} With the operation defined above, $G$ is a group isomorphic to $M_n$.
\end{lemma}
\begin{proof} We first need to check associativity of the operation. If $f\equiv 0$, we would obtain the Heisenberg group, thus for associativity, after cancellation, it remains to prove
\[ f(\vec{y}+\vec{y}', \vec{y}'') + f(\vec{y}, \vec{y}') =  f(\vec{y}, \vec{y}' + \vec{y}'') + f(\vec{y}', \vec{y}'')
\]
or equivalently
\begin{align*} &\floor{ \frac{(y_i + y_i')\bmod{p} + y_i'' \bmod{p}}{p} } +  \floor{ \frac{y_i \bmod{p} + y_i' \bmod{p}}{p} } \\
& = \floor{ \frac{y_i\bmod{p} + (y_i' + y_i'') \bmod{p}}{p} } +  \floor{ \frac{y_i'\bmod{p} + y_i'' \bmod{p}}{p} }.
\end{align*}
The expression 
\[
\floor{\frac{y_i \bmod{p} + y_i' \bmod{p} + y_i'' \bmod{p} }{p} } \in \{0,1,2\}
\]
can be checked to be equal to the left hand side. Comparing it with the right hand side it is enough to formally change the variables $y_i$ and $y_i''$. Both the left and right hand side count the largest multiple of $p$ less than or equal to 
\[
y_i \bmod{p} + y_i'\bmod{p} + y_i''\bmod{p}.
\]

Since $G$ is generated $\{[\vec{e_i}, 0, 0] , [0, \vec{e_i}, 0], [0,0,1]\}$, we define a homomorphism $\varphi: G \to M_n$ by $\varphi \left([\vec{e_i}, 0, 0]\right) = a_i$, $\varphi \left([0, \vec{e_i}, 0]\right) = b_i$, and $\varphi \left([0, 0, 1]\right) = c$. This map is clearly surjective and it is easy to check that the generators of $G$ satisfy the relations in $M_n$. Since $|G| = p^{2n+1}$, $\varphi$ is an isomorphism and $G \cong M_n$, as claimed.
\end{proof}

With this description, there is a natural way to define a brick in $M_n$. A subset $B$ of $M_n$ is said to be a \textit{brick} if
$$
B=\{[\underline{x}, \underline{y},z]\text{ such that }\underline{x}\in\underline{X},\   \underline{y}\in\underline{Y},\ z\in Z\}
$$
where $\underline{X}=X_1\times\dots\times X_n$ and
$\underline{Y}=Y_1\times\dots\times Y_n$ with 
non empty-subsets $X_i,Y_i, Z\subseteq\mathbb{F}_p$.

\subsection{Tools from spectral graph theory}
For a graph $G$ with vertex set $\{v_1,\dots, v_n\}$, the {\em adjacency matrix} of $G$ is the matrix with a $1$ in row $i$ and column $j$ if $(v_i, v_j)$ is an edge and a $0$ otherwise. Since this is a real, symmetric matrix, it has a full set of real eigenvalues. Let $\lambda_1 \geq \lambda_2 \geq \dots \geq \lambda_n$ be
the eigenvalues of its adjacency matrix. 

If $G$ is a $d$-regular graph, then its adjacency matrix has row sum $d$. In this case, $\lambda_1 = d$ with the all-one eigenvector $\mathbf{1}$.  Let $\mathbf{v}_{i}$ denote the corresponding eigenvector for $\lambda_i$. We will make use of the trick that for $i\geq 2$, $\mathbf{v}_i \in \mathbf{1}^{\bot}$, so $J\mathbf{v}_i = 0$ where $J$ is the all-one matrix of size $n \times n$ (see \cite{bh} for more background on spectral graph theory).

It is well-known (see \cite[Chapter 9]{as} for more details) that if $\lambda_2$ is much smaller than the
degree $d$, then $G$ has certain random-like properties. A graph is called {\em bipartite} if its vertex set can be partitioned into two parts such that all edges have one endpoint in each part. For $G$ be a bipartite graph with partite sets $P_1$ and $P_2$ and $U\subseteq P_1$ and $W\subseteq P_2$, let $e (U, W)$ be the number of pairs $(u, w)$ such that
$u \in U$, $w \in W$, and $(u, w)$ is an edge of $G$. We recall the following 
well-known fact (see, for example, \cite{as}).
\medskip

\begin{lemma}[Corollary 9.2.5, \cite{as}]\label{edge}
  Let $G = (V, E)$ be $d$-regular bipartite graph on $2n$ vertices with partite sets $P_1$ and $P_2$. For any two sets $B\subseteq P_1$ and  $C
  \subseteq P_2$, we have
  \[ \left| e (B, C) - \frac{d|B | |C|}{n} \right| \leq \lambda_2 \sqrt{|B| |C|}. \]
\end{lemma}

\subsection{Sum-product graphs}
Let $Q$ be a finite quasifield. The sum-product graph $SP_{Q, n}$ is defined as follows.  $SP_{Q,n}$ is a bipartite graph with its vertex set partitioned into partite sets $\mathbf{X}$ and $\mathbf{Y}$, where $\mathbf{X} = \mathbf{Y}= Q^n \times Q$. Two vertices $U = (\vec{x}, z) \in \mathbf{X}$ and
$V = (\vec{y}, z') \in \mathbf{Y}$ are connected by an edge, $(U, V)
\in E (SP_{Q, n})$, if and only if $\langle \vec{x}, \vec{y} \rangle=z+z'$. We need information about the eigenvalues of $SP_{Q,n}$.
\medskip

\begin{lemma}\label{sp-graph-lemma}
If $Q$ is a quasifield of order $q$, then the graph $SP_{Q,n}$ is $q^n$ regular and has $\lambda_2 \leq 2^{1/2}q^{n/2}$.
\end{lemma}

\medskip

We provide a proof of Lemma \ref{sp-graph-lemma} for completeness in the appendix, and we note that similar lemmas were proved in \cite{vinhajm} and \cite{PTTV}.

\section{Proof of Theorem \ref{thm1}}
\begin{lemma}\label{lm1}
Let  $B\subseteq M_n$ be a brick in $M_n$ with $B = [\vec{X}, \vec{Y}, Z]$ where $\vec{X} = X_1 \times \cdot \times X_n$ and $\vec{Y} = Y_1 \times\dots \times Y_n$. For given $\vec{a}=(a_1, \dots, a_n), \vec{b}=(b_1,\dots, b_n)\in \mathbb{F}_p^n$, suppose that \[|Z|^2\prod_{i=1}^n|X_i\cap (a_i-X_i)||Y_i\cap (b_i-Y_i)|> 2p^{n+2},\]
then we have 
\[B\cdot B \supseteq [\vec{a}, \vec{b}, \mathbb{F}_p].\]
\end{lemma}
\begin{proof}
Let $X_i'=X_i\cap (a_i-X_i)$, $Y_i'=Y_i\cap (b_i-Y_i), X'=(X_1',\dots, X_n')$, and $Y'=(Y_1',\dots, Y_n')$.
We first have 
\[B\cdot B\supseteq \{[\vec{x}, \vec{y}, z] \cdot [\vec{a}-\vec{x},\vec{b}-\vec{y}, z' ]~\colon \vec{x}\in X', \vec{y}\in Y', z, z'\in Z\}.\]
On the other hand, it follows from the multiplicative rule in $M_n$ that  for 
\[[\vec{x}, \vec{y}, z] \cdot [\vec{a}-\vec{x},\vec{b}-\vec{y}, z' ]=[\vec{a}, \vec{b}, z+z'+ \langle \vec{x}, (\vec{b}-\vec{y}) \rangle+f(\vec{y}, \vec{b}-\vec{y})].\]
To conclude the proof of the lemma, it is enough to prove that 
\[\left\lbrace z+z'+ \langle \vec{x}, (\vec{b}-\vec{y}) \rangle+f(\vec{y}, \vec{b}-\vec{y})\colon z, z'\in Z, \vec{x}\in X', \vec{y}\in Y' \right\rbrace=\mathbb{F}_p\]
under the condition $|Z|^2|X'||Y'|> 2p^{n+2}$.

To prove this claim, let $\lambda$ be an arbitrary element in $\mathbb{F}_p$, we define two sets in the sum-product graph $SP_{\FF_p, n}$, $E\subseteq \mathbf{X}$ and $F\subseteq \mathbf{Y}$ as follows:
\[E=X'\times (-Z+\lambda), ~F=\left\lbrace (\vec{b}-\vec{y},  -z-f(\vec{y}, \vec{b}-\vec{y}))\colon z\in Z, \vec{y}\in Y'\right\rbrace.\]
It is clear that $|E|=|Z||X'|$ and $|F|=|Z||Y'|$. It follows from Lemma \ref{edge} and Lemma \ref{sp-graph-lemma} that if $|Z|^2|X'||Y'|>2p^{n+2}$, then $e(E,F) > 0$. It follows that there exist $\vec{x}\in X', \vec{y}\in Y'$, and $z, z'\in Z$ such that 
\[z+z'+ \langle \vec{x}, (\vec{b}-\vec{y}) \rangle +f(\vec{y}, \vec{b}-\vec{y})=\lambda.\]
Since $\lambda$ is chosen arbitrarily, we have 
\[\left\lbrace z+z'+ \langle \vec{x}, (\vec{b}-\vec{y}) \rangle+f(\vec{y}, \vec{b}-\vec{y})\colon z, z'\in Z, \vec{x}\in X', \vec{y}\in Y' \right\rbrace=\mathbb{F}_p. \qedhere\]
\end{proof}

\paragraph{Proof of Theorem \ref{thm1}.}
We follow the method of \cite[Theorem 1.3]{HH}. First we note that if $|Z|>p/2$, then we have $Z+Z=\mathbb{F}_p$. This implies that 
\[B\cdot B =[2\vec{X}, 2\vec{Y}, \mathbb{F}_p ].\]
Therefore, $B\cdot B$ contains at least $|B\cdot B|/p\ge |B|/p$ cosets of the subgroup $[\vec{0}, \vec{0}, \mathbb{F}_p]$. Thus, in the rest of the proof, we may assume that $|Z|\le p/2$. 

For $1\le i\le n$, we have
\[
\sum_{a_i\in\mathbb{F}_p}|X_i\cap(a_i-X_i)|=|X_i|^2, \quad
\sum_{b_i\in\mathbb{F}_p}|Y_i\cap(b_i-Y_i)|=|Y_i|^2,
\]
which implies that
\[
\prod_{i=1}^n\left(\sum_{a_i\in\mathbb{F}_p}|X_i\cap(a_i-X_i)|\right)
\left(\sum_{b_i\in\mathbb{F}_p}|Y_i\cap(b_i-Y_i)|\right)
=\prod_{i=1}^n|X_i|^2|Y_i|^2.
\]
Therefore we obtain
\begin{equation}\label{eq1}
\sum_{\vec{a},\vec{b}\in\mathbb{F}_p^n}
\prod_{i=1}^n|X_i\cap(a_i-X_i)||Y_i\cap(b_i-Y_i)|
=\prod_{i=1}^n|X_i|^2|Y_i|^2.
\end{equation}
Let $N$ be the number of pairs $(\vec{a},\vec{b})\in\mathbb{F}_p^n\times \mathbb{F}_p^n$ such that
\[
|Z|^2\prod_{i=1}^n|X_i\cap(a_i-X_i)||Y_i\cap(b_i-Y_i)|>2p^{n+2}.
\]
It follows from Lemma \ref{lm1} that $[\vec{a}, \vec{b}, \mathbb{F}_p]\subseteq B\cdot B$ for such pairs $(\vec{a}, \vec{b})$. Then by equation \eqref{eq1} 
\[
\left(\prod_{i=1}^n|X_i||Y_i|\right)N+
2p^{n+2}(p^{2n}-N)>
\left(\prod_{i=1}^n|X_i||Y_i|\right)^2,\]
and so \[
N>\frac{\prod_{i=1}^n|X_i|^2|Y_i|^2-2p^{3n+2}}{\prod_{i=1}^n|X_i||Y_i|-2p^{n+2}}.\]

By the assumption of Theorem \ref{thm1}, we have
\begin{equation}\label{eq3}
|B|=|Z|\left(\prod_{i=1}^n|X_i||Y_i|\right)>|M_n|^{3/4+\varepsilon}=p^{3n/2+3/4+\varepsilon(2n+1)}.
\end{equation}
Thus when $n>1/\epsilon$, we have 
\[
\prod_{i=1}^n|X_i||Y_i|>p^{3n/2+7/4},
\]
since $|Z|\le p$.

In other words, 
\[N\ge (1-2p^{-3/2})\prod_{i=1}^n|X_i||Y_i|= (1-2p^{-3/2})\frac{|B|}{|Z|} \ge \frac{|B|}{p},\]
since $|Z|\le p/2$. \hfill $\qed$

\section{Proof of Theorem \ref{thm3}}
\begin{lemma}\label{lm3}
Let $Q$ be a quasifield of order $q$ and let $[\vec{X}, \vec{Y}, Z] = B\subseteq H_n(Q)$ be a brick. For a given $\vec{a} = (a_1,\dots, a_n)$, $\vec{b} = (b_1,\dots, b_n) \in Q^n$, suppose that 
\[
|Z|^2 \prod_{i=1}^n |X_i \cap (a_i - X_i)||Y_i \cap (b_i - Y_i) | > 2q^{n+2},
\]
then we have 
\[
B\cdot B \supseteq [\vec{a}, \vec{b}, Q].
\]
\end{lemma}

\begin{proof}
The proof is similar to that of Lemma \ref{lm1}, so we leave some details to the reader. Let
\[
X' = (X_1\cap (a_1 - X_1), \dots, X_n \cap (a_n-X_n)), ~Y' = (Y_1 \cap (b_1 - Y_1),\dots, Y_n \cap (b_n - Y_n)) 
\]
and $E\subseteq \mathbf{X}$, $F \subseteq \mathbf{Y}$ in $SP_{Q,n}$ where 
\[
E = X' \times (-Z + \lambda), ~F=\left\lbrace (\vec{b}-\vec{y},  -z)\colon z\in Z, \vec{y}\in Y'\right\rbrace,
\]
and $\lambda\in Q$ is arbitrary. Then $e(E,F) > 0$ which implies that there exist $\vec{x} \in X'$, $\vec{y} \in Y'$, and $z,z'\in Z$ such that 
\[
z+z' +\langle \vec{x}, (\vec{b} - \vec{y})  \rangle = \lambda.
\]
This implies that 
\[
[\vec{a}, \vec{b}, Q] \subseteq B\cdot B. \qedhere
\]
\end{proof}

The rest of the proof of Theorem \ref{thm3} is identical to that of Theorem \ref{thm1}. We need only to show that if $Z\subseteq Q$ and $|Z| > |Q|/2$, then $Z+Z = Q$. However, this follows since the additive structure of $Q$ is a group.

\section*{Appendix}

\begin{proof}[Proof of Lemma \ref{sp-graph-lemma}]
Let $Q$ be a finite quasifield of order $q$ and let $SP_{Q,n}$ be the bipartite graph with partite sets $\mathbf{X} = \mathbf{Y} = Q^n\times Q$ where $(x_1,\dots, x_n, z_x) \sim (y_1,\dots , y_n, z_y)$ if and only if
\begin{equation}\label{sp adjacency}
z_x+ z_y = x_1 * y_1 + \dots + x_n * y_n.
\end{equation}

First we show that $SP_{Q,n}$ is $q^n$ regular. Let $(x_1,\dots, x_n, z_x)$ be an arbitrary element of $\mathbf{X}$. Choose $y_1,\dots, y_n \in Q$ arbitrarily. Then there is a unique choice for $z_y$ that makes \eqref{sp adjacency} hold, and so the degree of $(x_1,\dots, x_n, z_x)$ is $q^n$. A similar argument shows the degree of each vertex in $\mathbf{Y}$ is $q^n$.

Next we show that $\lambda_2$ is small. Let $M$ be the adjacency matrix for $SP_{Q,n}$ where the first $q^{n+1}$ rows and columns are indexed by $\mathbf{X}$. We can write 
\[
M = \begin{pmatrix} 0 & N \\ N^T & 0 \end{pmatrix}
\]
where $N$ is the $q^{n+1} \times q^{n+1}$ matrix whose $(x_1 , \dots , x_{n}, x_z)_X \times (y_1 , \dots , y_{n},y_z)_Y$ entry is $1$ if \eqref{sp adjacency} holds and $0$ otherwise. 

The matrix $M^2$ counts the number of walks of length $2$ between vertices. Since $SP_{Q,n}$ is $q^n$ regular, the diagonal entries of $M^2$ are all $q^n$. Since $SP_{Q,n}$ is bipartite, there are no walks of length $2$ from a vertex in $\mathbf{X}$ to a vertex in $\mathbf{Y}$. Now let $x=(x_1,\dots, x_n, x_z)$ and $x'=(x_1',\dots, x_n', x_z')$ be two distinct vertices in $\mathbf{X}$. To count the walks of length $2$ between them is equivalent to counting their common neighbors in $\mathbf{Y}$. That is, we must count solutions $(y_1,\dots, y_n, z_y)$ to the system of equations 
\begin{equation}\label{adjacency 1}
x_z +y_z = x_1 * y_1 + \dots +x_n * y_n
\end{equation}
and
\begin{equation}\label{adjacency 2}
x_z' + y_z = x_1'  * y_1 + \dots + x_n'  * y_n.
\end{equation}

\noindent {\em Case 1: For $i\leq 1 \leq n$ we have $x_i = x_i'$}: In this case we must have $x_z \not = x_z'$. Subtracting \eqref{adjacency 1} from \eqref{adjacency 2} shows that the system has no solutions and so $x$ and $x'$ have no common neighbors.

\noindent {\em Case 2: There is an $i$ such that $x_i \not=x_i'$}: Subtracting \eqref{adjacency 2} from \eqref{adjacency 1} gives 
\begin{equation}\label{common neighbors}
x_z - x_z' = x_1 * y_1 + \dots + x_n  *y_n -x_1' *y_1 - \dots - x_n' *y_n.
\end{equation}
There are $q^{n-1}$ choices for $y_1,\dots, y_{i-1}, y_{i+1}, \dots y_n$. Since $x_i - x_i' \not=0$, these choices determine $y_i$ uniquely, which then determines $y_z$ uniquely. Therefore, in this case $x$ and $x'$ have exactly $q^{n-1}$ common neighbors.

A similar argument shows that for $y = (y_1,\dots, y_n, y_z)$ and $y' = (y_1',\dots, y_n', y_z')$, then either $y$ and $y'$ have either no common neighbors or exactly $q^{n-1}$ common neighbors.

Now let $H$ be the graph whose vertex set is $\mathbf{X}\cup\mathbf{Y}$ and two vertices are adjacent if and only if they are either both in $\mathbf{X}$ or both in $\mathbf{Y}$, and they have no common neighbors. For this to occur, we must be in Case 1, and therefore we must have either $x_z \not=x_z'$ or $y_z\not=y_z'$ and all of the other coordinates equal. Therefore, this graph is $q-1$ regular, as for each fixed vertex there are exactly $q-1$ vertices with a different last coordinate and the same entries on the first $n$ coordinates. Let $E$ be the adjacency matrix of $H$ and note that since $H$ is $q-1$ regular, all of the eigenvalues of $E$ are at most $q-1$ in absolute value. Let $J$ be the $q^{n+1}$ by $q^{n+1}$ all ones matrix. By the above case analysis, it follows that 
\begin{equation}\label{matrix squared}
M^2 = q^{n- 1} \begin{pmatrix} J & 0  \\ 0 & J \end{pmatrix}  + (q^n - q^{n-1} ) I - q^{n-1} E
\end{equation}

Now let $v_2$ be an eigenvector of $M$ for $\lambda_2$. For a set of vertices $Z$ let $\chi_Z$ denote the vector which is $1$ if a vertex is in $Z$ and $0$ otherwise (ie it is the characteristic vector for $Z$). Note that since $SP_{Q,n}$ is a regular bipartite graph, we have that $\lambda_1 = q^n$ with corresponding eigenvector $\chi_\mathbf{X} + \chi_{\mathbf{Y}}$ and $\lambda_n = -q^n$ with corresponding eigenvector $\chi_\mathbf{X} - \chi_{\mathbf{Y}}$. Also note that $v_2$ is perpendicular to both of these eigenvectors and therefore is also perpendicular to both $\chi_{\mathbf{X}}$ and $\chi_\mathbf{Y}$. This implies that
\[
\begin{pmatrix} J& 0 \\ 0 & J\end{pmatrix} v_2 = 0.
\]
Now by \eqref{matrix squared}, we have 
\[
\lambda_2^2 v_2 = (q^n - q^{n-1})v_2- q^{n-1}Ev_2.
\]
Therefore $q-1-\frac{\lambda_2^2}{q^{n-1}}$ is an eigenvalue of $E$ and is therefore at most $q-1$ in absolute value, implying that $\lambda_2 \leq 2^{1/2}q^{n/2}$. 

\end{proof}

\begin{thebibliography}{99}
\bibitem{as}
N.\ Alon and J.\ H.\ Spencer, \textit{The Probabilistic Method}, 2nd ed., Wiley-Interscience, 2000.
\bibitem{A}
M.\ Aschbacher,  \textit{Finite Group Theory}, Vol. 10. Cambridge University Press, 2000.
\bibitem{bh} A.\ Brouwer and W.\ Haemers, \textit{Spectra of Graphs}, Springer, New York, etc., 2012. 
\bibitem{D}
P.\ Diaconis, Threads through group theory, \textit{Character Theory of Finite groups}, \textit{Contemporary Mathematics}, \textbf{524} (2010): 33--47.

\bibitem{F} G.\ A.\ Freiman, Addition of finite sets, {\em Sov.\ Math.\, Dokl.} {\bf 5} (1964) 1366-1370.

\bibitem{G}
D.\ Gorenstein,  \textit{Finite Groups}, Vol. 301. American Mathematical Soc., 2007.

\bibitem{GR} B.\ Green and I.\ Ruzsa, Freiman's theorem in an arbitrary abelian group, {\em J.\ Lond.\ Math.\ Soc.} {\bf 75} (2007), 163-175.

\bibitem{HH2} N\ Hegyv\'ari and F.\ Hennecart, A note on Freiman models in Heisenberg groups, {\em Israel J.\ of Math.} {\bf 189} (2012), 397-411.

\bibitem{HH}
N.\ Hegyv\'{a}ri and F.\ Hennecart, A structure result for bricks in Heisenberg groups, \textit{Journal of Number Theory} \textbf{133}(9) (2013): 2999--3006.

\bibitem{PTTV} T.\ Pham, M.\ Tait, C.\ Timmons, and L.\ A.\ Vinh, A Szemer\'edi-Trotter type theorem, sum-product estimates in finite quasifields, and related results, {\em J.\ Combin. Theory Ser. A} {\bf 147} (2017) 55-74.

\bibitem{vinhajm}
L.\ A.\ Vinh, The solvability of norm, bilinear and quadratic equations over finite fields via spectra of graphs, \textit{Forum Mathematicum}, Vol. \textbf{26} (2014), No. 1, pp. 141--175.
\end{thebibliography}
\end{document}